\documentclass{article}

\usepackage{float}
\usepackage{xcolor}
\usepackage{graphicx}
\usepackage{hyperref}
\usepackage{amsfonts,amsmath,amsthm}
\usepackage{comment}

\theoremstyle{plain}
\newtheorem{theorem}{Theorem}[section]
\newtheorem{corollary}[theorem]{Corollary}
\newtheorem{lemma}[theorem]{Lemma}
\newtheorem{proposition}[theorem]{Proposition}

\theoremstyle{definition}
\newtheorem{example}[theorem]{Example}

\newtheorem{remark}[theorem]{Remark}

\newcommand\var{\operatorname{\mathbf{Var}}}

\newcommand{\indicator}{\mathbf{1}}

\newcommand{\R}{\mathbb{R}}
\newcommand{\N}{\mathbb{N}}
\newcommand{\M}{\mathcal{M}}
\newcommand{\X}{X}
\newcommand{\Y}{Y}

\renewcommand\P{{\operatorname{\mathbf{P}}}}   %===  %modifiado
\newcommand\E{{\operatorname{\mathbf{E}}}}

\newcommand{\D}{\mathbb{D}}

\begin{document}
\title{On the finiteness of the moments of the measure\\ of level sets of random fields}
\author{
  D. Armentano$^*$
%   \thanks{Universidad de la Rep\'{u}blica, Uruguay.   e-mail:diego@cmat.edu.uy}
\qquad
J-M. Aza\"is$^*$
% \thanks{Universit\'e de Toulouse, e-mail: jean-marc.azais@math.univ-toulouse.fr}
\qquad
F. Dalmao$^*$
% \thanks{Universidad de la Rep\'{u}blica, Uruguay. e-mail: fdalmao@unorte.edu.uy.}
\qquad
J. R. Le\'{o}n$^*$
% \thanks{Universidad de la Rep\'{u}blica, Uruguay and Universidad Central de Venezuela, Venezuela. e-mail: rlramos@fing.edu.uy}
\qquad
E. Mordecki\thanks{Armentano, Dalmao, Le\'on and Mordecki at Universidad de la Rep\'{u}blica, Uruguay. 
e-mails: diego@cmat.edu.uy, fdalmao@unorte.edu.uy, rlramos@fing.edu.uy, mordecki@cmat.edu.uy. 
Aza\"is at  IMT, Universit\'e de Toulouse, France and Universidad de la Rep\'{u}blica, Uruguay . 
e-mail: jean-marc.azais@math.univ-toulouse.fr. 
Le\'on at Universidad Central de Venezuela, Venezuela.
}
}

\maketitle

\begin{abstract}
General conditions on smooth real valued random fields are given to ensure the finiteness of the moments 
of the measure of their level sets.  
As a by product, a new generalized Kac-Rice formula (KRF) for the expectation of the measure of these level sets 
in the one-dimensional case is obtained 
when the second moment can be uniformly bounded. 
The conditions involve 
(i) the differentiability of the trajectories up to a certain order $k$, 
(ii) the finiteness of the moments of the $k$-th partial derivatives of the field up to another order and
(iii) the boundedness of the field's joint density and some of its derivatives. 
Particular attention is given to the shot noise processes and fields.
Other applications include stationary Gaussian processes, Chi-square processes and regularized diffusion processes. 
\end{abstract}

%\tableofcontents



%% main text
\section{Introduction}
Level sets of random fields play 
% saqué una s en plays
a key role 
in several branches of mathematics such as random algebraic geometry, probability, and mathematical physics. 
The focus is on different geometric characteristics of the level sets,  typically their geometric measure. 
Depending on the dimensions, these characteristics can be the number of crossings of a stochastic process, the length of level curves of a random field, etc.
Since the actual distribution of these random variables 
is usually out of reach, it is natural to study their moments 
and asymptotic distributions. 
 
The present work  concerns with the following two issues: 
(i) assessing the finiteness of the moments and 
(ii) their computation or estimation. 
For (ii) the main tool is the Kac-Rice formula (KRF for short) 
which gives the expected value and the higher order moments of
the measure of level sets of smooth random fields. 
These two issues are highly connected, for instance:

\begin{itemize}
\item[-]  The KRF of order one is, in general, valid under conditions that imply (without 
% cambié al plural
further hypotheses) that the expectation is finite. 
Furthermore, in some cases, necessary and sufficient conditions for the finiteness of the expectation of the measure of
level sets can be obtained from the KRF (see \cite{AAGL} and the references therein).

\item[-] For stationary Gaussian random fields, the KRF of order  two can be used to obtain the finiteness of the second moment of the measure of
level sets as in the works  \cite{geman}, \cite {kl}  and \cite{zazachichi}. Some complicated study  has been performed for higher moments in \cite{belyaev}. In other cases, the calculations are intractable.

\item[-] In the other direction, explained in section \ref{section:jma}, the  finiteness of the second moment is a tool to establish the 
validity of the KRF.

\end{itemize}
%
% historia
The seminal works of this field are that of  Kac \cite{kac} and Rice \cite{rice}. 
The KRFs were first established for Gaussian stochastic processes  
profiting from the fact that the Gaussian framework 
allows not only to obtain conditions under which 
the formulas are valid but also permits some explicit computations.
Adler in \cite{adler1} obtained KRF for Gaussian fields. 
Recent works of \cite{Poly} and \cite{AAGL} are also worth mentioning, as
they show the finiteness of the moments of the measure of nodal sets of a real-valued Gaussian stationary field. 
For a panoramic and contemporary view  of  these matters, 
we refer to the books \cite{adler}, \cite{aw} and \cite{BerzinEtAl}. 

% no Gaussiano
Mainly motivated by the applications, there has been an interest in
studying such formulas for non-Gaussian processes. The first successful
attempt was that of Marcus \cite{marcus} who provided a formula for the
expected number of crossings of a process whose trajectories are
absolutely continuous. 
Concerning finiteness of moments Nualart \& Wschebor  \cite{nw}, 
by using properties of the trajectories of regular processes, 
show that the expectation of the number of crossings and its moments of order greater than one can be bounded. 
This result is based on the idea that a ``nice function" cannot have too many zeros.  
Unfortunately, the proof heavily relies on the intermediate value theorem,
and for this reason, it applies only to stochastic processes. 
Wschebor in \cite{w-ln} establishes the KRF for the measure of the level sets of fields
(not necessarily Gaussian).
However, it is important to point out that the hypotheses of these formulas are difficult to check. 
An important exception is when the field is a nonlinear transformation of a Gaussian one, 
as is the case of $\chi^2$, $t$, or $F$ random fields. 
The books \cite{adler} and \cite{aw} contain a comprehensive update of these subjects.   
The papers \cite{bd0}, \cite{borokov} and \cite{dalmo} consider crossings for discontinuous processes, and
the two last works include KRFs. Moreover, Bierm\'e \& Desolneux (in \cite{bd} and \cite{bd2}) studied crossings problems and KRF for shot noise processes. 

% aplicaciones
Within the applications of the KRF we mention the random sea surface modeling, and
the articles \cite{mary}, \cite{podry}, \cite{baxpory} and \cite{podrywa} contain Gaussian and general KRFs.
Worsley computes the expectations of some level sets characteristic 
in the context of medical image processing  \cite{Worsley} and in astrophysics \cite{Worsley2}. 
Other applications and KRFs for fields can be found in the recent monograph \cite{BerzinEtAl}.

The present paper considers first finiteness of  moments which remains an open problem, excepting the stationary Gaussian case limited to the first two moments. 

It can be a first step to establish a speed of convergence in the ``ergodic'' case when we observe the random field over an increasing set of parameter.

In the particular case of $\eta$-dependent random processes (a random field $X$ is $\eta$-dependent if $X(t),X(s)$ are independent whenever 
$|t-s|>\eta$), the finiteness of the second moment  gives directly a central limit theorem.

This finiteness gives also a central limit theorem in the case of increasing number of independent observations of the random field.

%of  the measure of level sets 
%and with non Gaussian KRFs. 
%

The main result in dimension one is given in section \ref{section:main}. 
%We present a generalization of the result of \cite{nw} with, in our opinion, a simpler proof. 
Section \ref{section:examples} considers the application of the previous result to different examples with new results, even in the case of Gaussian processes, or $\chi^2$-processes.
Section \ref{s:sn1} contains the study of shot noise processes, this theme constitutes, together  with shot noise random fields, the main application of our results. 
In section \ref{section:co-one}, the previous result is extended to level sets of co-dimension one
that corresponds to random fields from $\R^d$ and the sphere $S^d$ to $\R$. 
Section \ref{snp} deals with shot noise random fields. 
Finally,  non-Gaussian KRFs are in section \ref{section:jma}. 
All results  presented in the examples are new  but, except for the toy Example 1, impossible  it is to know if they are optimal or not.

\section{Processes from $\R$ to $\R$}\label{section:main}
% % intro

The basic idea of this section is the following (see the details below): \\ 
Let $f(\cdot)$ be defined as, say, $[0,1]$. 
Assume that the $k$-th derivative of $f$ is bounded by $M$ and 
that $f$ has $k$ zeros on $[0,1]$. 
Then, $f$ satisfies 
\[
|f(1/2)| \leq  M \frac {(1/2) ^k}{k !}.
\]

Now, if we replace the function $f$ by the paths of a random process $X$ that admits a density at $t=1/2$ 
and we assume that this density is bounded by $C$. 
Then, 
the probability of the event $\{|X(1/2)| \leq  M \frac {(1/2) ^k}{k !}\}$ 
is  smaller than
\[
2C M \frac {(1/2) ^k}{k !}.
\]

Theorem \ref{theorem:key} is just a systematic exploitation  of this method  with some generalization  because  we consider the  joint density of 
$X(t)$ and some derivatives. \bigskip

% the process
Assume that $\X=\{X(t)\colon t\in\R\}$ is a real valued stochastic process
with smooth paths. 
% crossings
Define the number of crossings through level $u$ by the process $\X$ over the finite time interval $I\subset\R$ 
by
\begin{equation*}
N_u=N_u(I):=\#\{t\in I\colon X(t)=u\}.
\end{equation*}
Let  $|I|$ denote the length of $I$ and let  $\bar{I}$ be the middle point of $I$. 

As usual, we interpret the zeroth derivative $X^{(0)}$ as $X$ itself.
Thus, we have the following result.

% main results
\begin{theorem}\label{theorem:key}
Consider  $u$,  $\X$, $I$ and $N_u$ as above. 
Assume that $\X$ satisfies 
\begin{itemize}
\item[\rm(H1)] The sample paths of $\X$ are $C^k(I)$ for some  $k\geq 2$.
\item[\rm(H2)] For some  $m=1,2,\dots$, there exists a constant $D_m$ such that
$$
\E\left({|}X^{(k)}{|}_{\infty}^m\right)\leq D_m.
$$
\item[\rm(H3)] There exist $0\leq h\leq k$ and  a constant $C>0$ such that
the joint density of $X(t),X'(t),\dots,X^{(h)}(t)$ is bounded by $C$ uniformly in $t\in I$, 
and on  a neighborhood of $(u,0,\dots,0)$.
\end{itemize}
Then, for $p=1,2,\dots$, such that 
\begin{equation}\label{eq:nw}
p<\left(k-\frac{h}2-{1\over1+h}\right)\left(\frac1m+\frac1{1+h}\right)^{-1},
\end{equation}
the $p$-th moment of the number of  crossings $N_u$ is finite.   
 If in addition we assume \\
{\rm (H4)} the density in {\rm(H3)} is bounded by $C$ uniformly 
 in $t\in I$ and $(u_1,\ldots, u_{h+1 }) \in \R^{h+1}$. \\
 Then
$$
\E(N_u^p)\leq  (k-1)^p+D_m\cdot E_{\alpha,k,p}+C|I|^{(h+1)(k-h/2)}\cdot D_{\alpha,k,h,p},
$$
where $\alpha$ is any real number such that 
$\frac{p}m<\alpha<k-\frac{h}2-{1+p\over1+h}$, 
\[
{\color{black}
E_{\alpha,k,p}=p(k-1)^{p-1}\sum^\infty_{a=1}\frac{(a+1)^{p-1}}{a^{m\alpha}}  }
\quad\textrm{ and }\quad
D_{\alpha,k,h,p}=
\frac{2^{(h+1)(1+h/2-k)}}{k!\cdot(k-h)!}\sum^\infty_{a=1}(a+1)^{p-1}a^{1-(h+1)(k-h/2-\alpha)}.
\]

\end{theorem}

\begin{remark}[Large $m$]
The limit case, when $m$ can be chosen arbitrarily large (as is the case for Gaussian processes), 
corresponds to $\alpha$ close to zero.
The limit condition \eqref{eq:nw} then becomes
$$
p<\left(k-\frac{h}2\right)(h+1)-1.
$$
\end{remark}

\begin{remark}[Large $k$]
When the process $X$ has $C^{\infty}(I)$ paths, 
 if {\rm(H2)} holds true for $m=1$ with arbitrary $k$, 
 and {\rm(H3)} holds true for some $h=0,1,\dots$,
then, all the moments of $N_u$ are finite.
\end{remark}

\begin{remark}[$h=0$]
The case $h=0$ and $m=1$ corresponds to (3.23) in pg. 82 in \cite{aw}, obtained in Lemma 1.2 in \cite{nw}.
When $h=0$ and $m>1$ Theorem \ref{theorem:key} assures that the $p$th moment is finite for
$$
p<{m\over m+1}(k-1).
$$
\end{remark}

Before proving the theorem we establish three  preliminary lemmas. 
{\color{black}Recall that $f^{(0)}=f$.}
\begin{lemma}\label{lemma:lagrange}
Consider a function $f\colon I\to\R$ of class $C^k(I)$ for an interval $I$ 
{\color{black}and $k\geq 1$}. 
If $f-u$ has $k$ roots in $I$ and {\color{black}$0\leq$}$h\leq k$, we have
\[
|f(\bar{I})-u|\leq {{|}f^{(k)}{|}_{\infty}\over k!}\left({|I|\over 2}\right)^k,\
|f'(\bar{I})|\leq {{|}f^{(k)}{|}_{\infty}\over (k-1)!}\left({|I|\over 2}\right)^{k-1},\
\dots,\
|f^{(h)}(\bar{I})|\leq {{|}f^{(k)}{|}_{\infty}\over (k-h)!}\left({|I|\over 2}\right)^{k-h}.
\]
\end{lemma}
\begin{proof}
The proof is based on the Lagrange remainder form for polynomial interpolation. 
That is, let $g:I\to\R$ be $C^k$, $t_1,\dots,t_k\in I$ and $P$ be the only 
polynomial of degree $k-1$ 
such that $g(t_j)=P(t_j),j=1,\dots,k$. Then, for $t\in I$ we have
$$
g(t)-P(t)=\frac{1}{k!}\prod^k_{j=1}(t-t_j)\cdot g^{(k)}(\xi),
$$
for some $\xi$ such that $\min\{t_1,\dots,t_k,t\}<\xi<\max\{t_1,\dots,t_k,t\}$, 
see Lemma 5.2, p. 135 in \cite{aw}. 

Now, for $g=f-u$, $t=\bar{I}$ and using as $t_1,\dots,t_k$ 
the roots of $f-u$, we get  {\color{black} $P=0$ }and so the first inequality follows. 
The proofs of the other inequalities follow in a similar way.
\end{proof}
\begin{lemma}\label{lemma:fubini}
Let $Z$ be a random variable taking non negative integer values. Then,
\begin{equation*}
\E(Z^p)=\sum_{\ell=1}^\infty\left(\ell^p-(\ell-1)^p\right)\P(Z\geq \ell)\leq
p\sum_{\ell=1}^\infty\ell^{p-1}\P(Z\geq \ell).
\end{equation*}
\end{lemma}
\begin{proof} 
Use Fubini's Theorem and the convexity of the function $x\mapsto x^p$.
\end{proof}
\begin{lemma}\label{lemma:3}
Assume that conditions {\color{black}{\rm (H1), (H2), (H3)} and {\rm(H4)}} of Theorem \ref{theorem:key} hold true.
Let $\ell\geq k$ and define $a$ and $r$ such that $\ell=a(k-1)+r$ with $1\leq r\leq k-1$.
Then, for $B>0$, we have:
\begin{equation*}
\P(N_u\geq \ell,{|}X^{(k)}{|}_\infty\leq B)
\leq aC(2B)^{h+1}\prod_{i=0}^h\left({|I|\over 2a}\right)^{k-i}\frac1{(k-i)!}.
\end{equation*}
where $C$ is the bound in {\rm(H3)}.

  {\color{black} When {\rm(H4)} does not  hold, the inequality is true for $ \ell$ sufficiently large only.}
\end{lemma}
\begin{proof}
We divide $I$ into $a$ equal subintervals   {\color{black}$I_1, \ldots  I_a$}, s.t. at least one of them has $k$ zeros.
After this, we use an union bound: 
\begin{align*}
\P(N_u\geq\ell,{|}X^{(k)}{|}_\infty\leq B)&=
\P(N_u\geq a(k-1)+r,{|}X^{(k)}{|}_\infty\leq B)
\leq \P\left(\bigcup^{a}_{j=1}\{N_u(I_j)\geq
  k\},{|}X^{(k)}{|}_\infty\leq B\right)\\
&\leq
\sum^a_{j=1}\P\left(|X(\bar{I}_j)-u|\leq \frac{B}{k!}\left(\frac{|I|}{2a}\right)^k,
\bigcap^h_{i=1}|X^{(i)}(\bar{I}_j)|\leq \frac{B}{(k-i)!}\left(\frac{|I|}{2a}\right)^{k-i}\right)\\
&\leq aC\prod_{i=0}^h\frac{2B}{(k-i)!}\left(\frac{|I|}{2a}\right)^{k-i},
\end{align*}
obtaining the result.  {\color{black} Of course,  when $ h=0$, the intersection in  the equation above   must be absent.}
\end{proof}

\begin{proof}[Proof of Theorem \ref{theorem:key}]  {\color{black}Assume (H4) for the moment }
From Lemma \ref{lemma:fubini}, we have
\begin{equation}\label{eq:np}
\E(N_u^p)
\leq  (k-1)^p+\sum_{\ell=k}^\infty p\ell^{p-1}\P(N_u\geq \ell).
\end{equation}
In order to bound the summands in the r.h.s. of \eqref{eq:np}, for given $\ell$ consider $a$ and $r$ as in Lemma \ref{lemma:3},
and $\alpha>0$ to be defined later.
We use the following bound:
$$
\P(N_u\geq \ell)\leq\P(N_u\geq \ell,{|}X^{(k)}{|}_\infty\leq a^\alpha)+\P({|}X^{(k)}{|}_\infty> a^\alpha).
$$
Using now Lemma \ref{lemma:3} for the first summand and  {\color{black}Chebyshev's} inequality for the second,
we obtain 
\begin{multline*}
\E(N_u^p)\leq (k-1)^p+p(k-1)^{p-1}\sum_{a=1}^\infty(a+1)^{p-1}\P(N_u\geq \ell)\\
\leq (k-1)^p+p(k-1)^{p-1}\left(
C2^{h+1}\left(\frac{|I|}{2}\right)^{(h+1)(k-h/2)}\prod_{i=0}^h\frac{1}{(k-i)!}.
\sum_{a=1}^\infty(a+1)^{p-1}a^{(h+1)(\alpha+h/2-k)+1}\right.\\
\left.
+D_m\sum_{a=1}^\infty(a+1)^{p-1}a^{-m\alpha}\right).
\end{multline*}
The conditions for the simultaneous convergence of the two series above are
\begin{equation*}
\frac{p}m<\alpha<k-\frac{h}2-{1+p\over1+h}.
\end{equation*}
Under condition \eqref{eq:nw}, there always exists an adequate $\alpha$, concluding the proof of Theorem \ref{theorem:key}. 

 {\color{black}   When (H4) does not hold, the calculation above is true only for $\ell$ (or $a$) sufficiently large, and this does not change the 
condition of convergence of the series. }
\end{proof}

\section{Examples}\label{section:examples}
\begin{example}[{\bf Conditional sine-cosine process}]
We begin with a toy example. Consider a random variable $\omega$ such that $\E\ (|\omega|^M)<\infty$ and $\E\ (|\omega|^{M+1})=\infty$. 
This variable is 
the random frequency of a process 
$\X=\{X(t)\colon 0\leq t\leq 2\pi\}$ that we construct as
$$
X(t)=
\xi_1\sin\omega t+\xi_2\cos\omega t,
$$
where $\xi_1$ and $\xi_2$ are two standard normal independent random variables, also independent from $\omega$.
Conditionally to the value of $\omega$, the process $\X$ is a Gaussian sine-cosine process.
It is clear that the density of $X(t)$ is standard normal for each $t$.
Nevertheless, it should be observed that $X(t)$ is, in general, not a Gaussian process.
The number $N_0$ of roots of $X(t)$  {\color{black} on $[0,2\pi]$ }coincides with the number of roots of
the process
$$
\tilde{X}(t)=
\frac{\xi_1}{\sqrt{\xi_1^2+\xi_2^2}}\sin\omega
t+\frac{\xi_2}{\sqrt{\xi_1^2+\xi_2^2}}\cos\omega t=\cos(\omega t-\theta),
$$
where $\theta=\arctan(\xi_1/\xi_2)$ is defined a.s. 
Therefore, $N_0$ satisfies  
$|N_0-2\omega|\leq 2$.
This means that for the considered process, the maximal finite moment of $N_0$ is $M$.
To apply Theorem \ref{theorem:key}, we compute 
$$
{|}X^{(k)}{|}_\infty\leq\omega^k\sqrt{\xi_1^2+\xi_2^2}.
$$
Note that $\sqrt{\xi_1^2+\xi_2^2}$ follows a Rayleigh's distribution.
Then, based on the independence of the three random variables above, we have
$$
\E\left({|}X^{(k)}{|}_\infty^m\right)\leq H_m\E(\omega^{mk}),
$$
where $H_m$ is the finite moment of the Rayleigh distribution.
So, we can apply Theorem \ref{theorem:key} with $h=0$, $k=M$ and $m=1$. 
Assuring that the moments of order $p\leq (M-1)/2$ are finite. 
This result should be compared 
with the true result which is $M$.
\end{example}

\begin{example}[{\bf Stationary Gaussian process}] \label{ex:stat}
Consider a stationary Gaussian process with spectral measure 
supported in a set with an accumulation point. {\color{black} By Ex. 3.4 in \cite{aw} we know that for all $h$ 
$X(t), X'(t), \ldots , X^{(h)}(t)$ admit a joint density. 
By Borell-Sudakov-Tsirelson Theorem (see, for example Th. 2.8 in  \cite{aw}), for all $m$:
\[
\E\left({|}X^{(k)}{|}_{\infty}^m\right)\leq + \infty.
\]
It is then possible to take $h=k$ and $m$ arbitrarily large.} 
Thus, Formula \eqref{eq:nw}
gives 
a finite moment of order
\[
p {\color{black} \leq}{k(k+1)\over 2}-2. 
\]
To visualize it, see Table \ref{table:max}.
\begin{table}[H]
\begin{center}
\begin{tabular}{|c|c|c|c|}
\hline 
$k$ & $2$ & $3$ & $4$\\
\hline
$p$ & 1 &  2,3,4 & 5,6,7,8\\
\hline
\end{tabular}
\caption{If $\X$ is $C^k$, following Theorem \ref{theorem:key} we obtain that $\E(N_u^p)$ is finite for values of $k$ and $p$ above.}\label{table:max}
\end{center}
\end{table}
This result improves the previous one obtained in \cite{nw}
that is recovered in the case $h=0$ and $m=1$. 
\end{example}

\begin{example}[{\bf Chi-square process}]
Consider a Chi-square process
$\Y=\{Y(t)\colon t\in\R\}$ given by 
$$
Y(t)={|}X(t){|}^2=\sum_{i=1}^nX_i(t)^2,
$$
where $X=\{X(t)\colon t\in\R\}$ with $X(t)=\left(X_1(t),\dots,X_n(t)\right)$ is an $\R^n$ valued random process whose coordinates 
are $n$ independent copies of  a stationary Gaussian process with $C^k(\R)$ paths and variance 1.
First, observe that the case $n=1$ admits an ad-hoc treatment. If $n=1$
for $u=0$ we have $N^Y_0[0,T]=N^X_0[0,T]$, 
and  for $u>0$ we have $N^Y_u[0,T]=N^X_{\sqrt{u}}[0,T]+N^X_{-\sqrt{u}}[0,T]$.
For $n\geq 2$, the level $u=0$ is not interesting, as for this level the process is differentiable and
non-negative and has a bounded density, 
so Bulinskaya  Lemma (Prop. 1.20 in  \cite{aw}) gives that a.s. the number of crossings is zero.
Consequently, we consider $n\geq 2$ and $u>0$.
Observe that, excluding the uninteresting case where $\lambda_2=\var X'(t)=0$, 
the joint density of the random variables
$X(t),X'(t)$ is bounded, as they conform a pair of independent Gaussian random variables. 

It is not difficult to check the hypothesis of Theorem \ref{theorem:key} in two different situations: $n=2$ and $n\geq 3$. 
It is direct to see that $\Y$ has $C^k$ paths and that $ {|}\Y{|}_{\infty}$ has moments of every order
when considering a finite interval $I\subset\R$.
Observe now that $Y(0)$ has a $\chi^2(n)$ density that is bounded for $n\geq 2$.
We can then apply Theorem \ref{theorem:key} with the given $k$, arbitrary $m$, and $h=0$.
Based on \eqref{eq:nw}, we obtain the finiteness of the moments of the crossings, for $u>0$, of order
\[
p= k-2,
\] 
 {\color{black} which is relevant only in the case $k \geq 3$.}
A more refined analysis can be carried out.
Regarding the derivative, we have
\[
Y'(t) = 2\sum_{i=1}^nX_i(t)X'_i(t).
\]
We see that conditional to $X(t)$, the random variable $Y'(t)$ has a Gaussian distribution with variance $4\lambda_2 Y(t)$.
Its conditional density $p_{Y'|X=x}$  is bounded by
\[
(Const)  \big(Y(t)\big)^{-1/2},
\]
where $(Const)$ denotes a meaningless constant whose value may change from line 
to line. 

Let now $B_1$ and $B_2$ be two Borel sets of $\R$ with respective  measures $|B_1|$ and $|B_2|$. Writing $Y$ and $Y'$ for  $Y(t)$ and $Y'(t)$
\begin{align} \label{e:jma:chi}
\P(Y' \in B_1, Y \in B_2) 
& = \int_{\R^n} dx  \int_{B_1}   p_{Y'|X=x}(y)  \indicator_{\{Y \in B_2\}} p_X(x) dy  \notag \\
& \leq 
 (Const) \int_{\R^n}   |B_1| Y^{-1/2} \indicator_{\{Y \in B_2\}} p_X(x) dx \notag \\
& = (Const)  |B_1|   \int _{B_2}   y^{-1/2}  p_Y(y) dy = (Const)  |B_1|    \int _{B_2}  y^{n/2-3/2}  e^{-y/2}   dy,  
\end{align}
where $p_X(.)$ and $ p_Y(.)$ are the density functions of $X$ and $Y$, respectively. 
Note that we used the explicit expression of the density of a $\chi^2(n)$ 
distribution. 
When $n \geq 3$, the  integrand  in \eqref{e:jma:chi} is bounded yielding that 
$$
\P\left(Y' \in B_1, Y \in B_2\right) \leq (Const) |B_1||B_2|.
$$
This proves that the process satisfies the hypotheses  of Theorem \ref{theorem:key} with $h=1$.
The conclusion is that for $n\geq 3$, the number of crossings of $Y$  has a finite moment of order
\begin{equation*}%\label{eq:chi}
p =2k -3.
\end{equation*}
%{\color{black} which is relevant only in the case $k \geq 2$.}{\color{black} ya dicho !}
\end{example}

\begin{example}[{\bf Regularized processes}]
In this example, we consider the number of crossings with a level $u$ 
of a regularized diffusion. We depart from a diffusion  $\X=\{X(t):t\geq 0\}$ 
defined as the solution of the stochastic differential equation
\begin{equation*}
X(t)=x_0+\int^t_0b(s,X(s))ds+
\int^t_0\sigma(s,X(s))dW(s),\quad t\geq 0.
\end{equation*}
Here, $W=\{W(s):s\geq 0\}$ is a standard Brownian motion, $x_0\in\R$, and 
$b,\sigma:[0,\infty)\times\R\to\R$ are Lipschitz w.r.t. the second variable, 
that is, for $T>0$
there exists $K_T>0$ such that 
for $x,y\in\R$ and $0\leq s\leq T$ it holds that 
\begin{equation*}
|b(s,x)-b(s,y)|+|\sigma(s,x)-\sigma(s,y)|\leq K_T|x-y|.
\end{equation*}
Assume also the linear growth condition
$$
|b(s,x)|+|\sigma(s,x)|\leq K_T(1+|x|).
$$
Then, the above equation has a unique strong solution that satisfies
\begin{equation*}%\label{eq:sde}
\sup_{0\leq t\leq T}\E(X(t)^2)\leq H_T<\infty.
\end{equation*}
for a constant $H_T$ (see Theorem 7.1 in \cite{Borodin}).
The regularized diffusion is defined by
\begin{equation*}
X_\Psi(t)=(\Psi\ast X)(t),
\end{equation*}
where $\Psi\colon\R\to[0,\infty)$ is a $C^\infty$-function
with support contained in $[-1,1]$ that integrates   to one.   
Thus, $X_\Psi$ is obtained from $\X$ by path-wise convolution with $\Psi$.
We consider two cases:
\begin{enumerate}
\item[Case I.] We assume that $b=0$ and 
$\sigma:\R^+\times\R\to\R$ is strictly positive and $C^3$.
\item[Case II.] The volatility $\sigma$ is strictly positive and $C^3$ as in Case I, and there exist constants $B_T, $$c_T$ and $C_T$ such that 
$$
|b(s,x)|\leq B_T,\quad 0<c_T\leq\sigma(s,x)\leq C_T
$$ 
for all $0\leq s\leq T$ and $x\in\R$.
\end{enumerate}
We start with Case I. 
With regards to the hypotheses of Theorem  \ref{theorem:key}, we have that   
$X_\Psi$ is $C^\infty$ since it inherits the regularity of $\Psi$.
Furthermore, for $a$ large enough (depending on $\Psi$) 
 the random variable $X_\Psi(t)$ 
has a uniformly bounded density on an interval $[a,T]$ for all finite $T$, see Lemma 3.1 in \cite{nw}. 
Besides, a direct computation gives
$$
{|}X^{(h)}_\Psi{|}_\infty\leq c\cdot {|}X{|}_\infty,
$$
with $c=\int|\Psi^{(h)}(u)|du$ and the infinity norm is taken on the interval $[a,T]$.
Hence, it suffices to bound the moments of ${|}X{|}_\infty$. 
Now, the Burkholder-Davis-Gundy inequality {\color{black} \cite[Th.48]{protter} }gives
$$
\E({|}X{|}^2_\infty)\leq C_1 \E([X]_T),
$$
 {\color{black} with}  $[X]_T$ the quadratic variation of $X$ on $[0,T]$. 
Now, by Theorem II-29 in Protter \cite{protter}, we have
\begin{align*}
[X]_T&=\left[\int^T_0\sigma(s,X(s))dW(s)\right]_T=\int^T_0\sigma(s,X(s))^2ds\\
  &\leq K_T^2\int^T_0(1+|X(s)|)^2ds\leq
  2K_T^2\int^T_0(1+|X(s)|^2)ds.
\end{align*}
Taking expectations, we have
\begin{align*}
\E([X]_T)\leq 2K_T^2\int^T_0(1+\E(X(s))^2)ds\leq 2K_T^2T(1+H_T)<\infty.
\end{align*}
Hence, by Theorem \ref{theorem:key} with $h=0$, $m=2$ and arbitrary $k$, we obtain that 
$$
\E(N_u^p)<\infty,
$$
for all $p$, as obtained in \cite{nw}.

For Case II, we apply Girsanov's theorem \cite[Th. 10.1-10.2]{Borodin}. 
We then have two SDE
\begin{align*}
dX_0(t)&=\sigma(t,X_{\color{black}0}(t))dW(t)\\
dX(t)&=b(t,X(t))dt+\sigma(t,X(t))dW(t).
\end{align*}
Consider the process density:
$$
\rho(t)=\exp\left(
\int_0^t{b(s,X(s))\over\sigma(s,X(s))}dW(s)-\frac12\int_0^t\left({b(s,X(s))\over\sigma(s,X(s))}\right)^2ds
\right).
$$
Then, Girsanov's Theorem states that for an arbitrary function
on the trajectories of the process $F\colon C([0,T],\R)\to\R$  we have
$
\E\left(F(X)\right)=\E\left(\rho(T)F(X_0)\right).
$
If $F$ is the $p$-power of the number of crossings on $[0,T]$, we have
\begin{align*}
\E\left(
N_u(I,X)^p
\right)
=
\E\left(
\rho(T)
N_u(I,X_0)^p
\right)\leq 
\E\left(
\rho(T)^2
\right)^{1/2}\E\left(
N_u(I,X_0)^{2p}
\right)^{1/2},
\end{align*}
applying the H\"older's inequality ($X$ and  $X_0$ are defined above).
To bound the first expectation in the r.h.s. consider
$$
\rho(t)^2=\exp\left(
2\int_0^t{b(s,X(s))\over\sigma(s,X(s))}dW(s)-2\int_0^t\left({b(s,X(s))\over\sigma(s,X(s))}\right)^2ds
\right)
\exp\left(
\int_0^t\left({b(s,X(s))\over\sigma(s,X(s))}\right)^2ds
\right)
$$
Then,
$$
\rho(t)^2\leq \exp\left(
2\int_0^t{b(s,X(s))\over\sigma(s,X(s))}dW(s)-2\int_0^t\left({b(s,X(s))\over\sigma(s,X(s))}\right)^2ds
\right)
 \exp\left(\left({B_T\over c_T}\right)^2T\right),
$$
and
{\color{black} since the first exponential is a martingale, we get}
$$
\E\rho(T)^2\leq 
 \exp\left(\left({B_T\over c_T}\right)^2T\right)<\infty.
$$
As, {\color{black} for all $p$},  $\E\left(N_u(I,X_0)^{2p}\right)$ is finite by Case I, the finiteness of 
$\E\left(N_u(I,X)^p\right)$ follows. 
\end{example}

%%%%%%%%%%%%%%%%%%%%%%%%%%%%%%%%%%%%%%%%%%%%%%%%%%%
\section{Shot noise processes} \label{s:sn1}
The stationary shot noise process $\X=\{X(t)\colon t\in\R^d\}$ is defined by
\begin{equation}\label{eq:sn} 
X(t)=\sum_i\beta_ig(t-\tau_i),
\end{equation} 
where $(\beta_i)$ {\color{black} is}  a sequence of i.i.d. random, the ``impulse",
variables, $(\tau_i)$ is a Poisson field on $\R^d$ with
 constant intensity $\lambda$, and $g:\R^d\to\R$ is  some function called
the kernel function.  Following \cite{bd}, we assume that $\beta_1$ is
an integrable random variable and that $g\in L^1(\R^d)$. This ensures the a.s. convergence of
the series in \eqref{eq:sn}.  
A key issue to apply Theorem \ref{theorem:key} is the verification of condition  {\color{black} (H4)} with $h=0$,
which requires the boundedness of a density.
This is a delicate issue as was previously noticed by Bierm\'e and Desolneux in \cite{bd} (see also \cite{bd1}),
that is the main reference of this section from where we borrow the presentation and notations.  
We consider below $d=1$, 
the case $d>1$ (as well as shot noise random fields defined on the sphere) 
will be considered in Section \ref{snp}. 
We begin by  specializing Theorem \ref{theorem:key} to the present situation.
\begin{corollary}\label{thm:shotnoiseR}
Consider a stationary shot noise process \eqref{eq:sn} satisfying condition 
{\rm(H1)} for some $k\geq 1$,
{\rm(H2)} for some $m\geq 1$ and $k$ above, and 
{\rm(H3)} for $h=0$.
Then,
$$
\E (N_u^p)<\infty
$$
for 
\begin{equation}\label{eq:pnw}
p<{m\over m+1}(k-1).
\end{equation}
\end{corollary}
The rest of the section is devoted to obtain sufficient conditions to verify this corollary.
The differentiability of the sample paths follows directly from the differentiability of the kernel $g$:
$$
X^{(k)}(t)=\sum_i\beta_ig^{(k)}(t-\tau_i),
$$
provided that the $k$-th derivative of $g$ is integrable, i.e. $g^{(k)}\in L^1(\R)$.\\

\noindent{\bf Boundedness of moments (H2).} 
We now give conditions on the impulse and the kernel in order to verify (H2) for given  $k$ and $m$.
For simplicity of exposition we assume that $I=[-1,1]$ and the general case can be treated in the same way.
We consider the partition of the real line $I_1=I=[-1,1]$ and $I_n=[-n,-n+1)\cup(n-1,n]\ (n\geq 2)$. 
In this way we can write
$$
X^{(k)}(t)=\sum_{n=1}^{\infty}\sum_{i:\tau_i\in I_n}\beta_ig^{(k)}(t-\tau_i).
$$
\begin{proposition}\label{p:dkn}
Consider the shot noise in \eqref{eq:sn}
with impulse such that $\E(|\beta_1|^m)<\infty$ for given $m\geq 1$.
Define  
$$
d_{k,n}=\sup_{t\in I_1,s\in I_n}|g^{(k)}(t-s)|,
$$
for $k\geq 1$,
and  assume that
\begin{equation}\label{eq:dk}
D_{k}=\sum_{n=1}^{\infty}d_{k,n}<\infty.
\end{equation}
Then, the condition {\rm(H2)} holds true for $m$ and $k$ as above, i.e.
$$
\E\left({|}X^{(k)}{|}_{\infty}^m\right)<\infty.
$$
\end{proposition}
\begin{proof}
We have
\begin{align*}
  {|}X^{(k)}{|}_{\infty}=&\max_{-1\leq t\leq 1}|X^{(k)}(t)|
\leq
\max_{-1\leq t\leq 1}\sum_{n=1}^{\infty}\sum_{i:\tau_i\in I_n}|\beta_i||g^{(k)}(t-\tau_i)|\\
  &\leq \sum_{n=1}^{\infty}\sum_{i:\tau_i\in I_n}|\beta_i|\max_{-1\leq t\leq 1}|g^{(k)}(t-\tau_i)|
\leq \sum_{n=1}^{\infty}\sum_{i:\tau_i\in I_n}|\beta_i|\sup_{-1\leq t\leq 1,s\in I_n}|g^{(k)}(t-s)|\\
  &=\sum_{n=1}^{\infty}d_{k,n}Z_n,
\end{align*}
where $Z_n:=\sum_{\tau_i\in I_n}|\beta_i|$. 
As the sets $I_n$ are disjoint and have the same length, the random variables
$Z_n$ are i.i.d. Each one has a compound Poisson distribution.
As $\E(|\beta_1|^m)<\infty$, we obtain 
$\E(Z_1^m)<\infty$.
We now use Jensen's inequality, based on the convergence in \eqref{eq:dk}, to obtain the following bound:
\[
  \left(\sum_{n=1}^{\infty}d_{k,n}Z_n\right)^m=
  D_{k}^m\left(\sum_{n=1}^{\infty}{d_{k,n}\over D_k}Z_n\right)^m \leq D_k^{m-1}\sum_{n=1}^{\infty}d_{k,n}Z_n^m.  
\]
Now, as the random variables $Z_n^m$ are i.i.d. and have finite moments,
\[
\E\left({|}X^{(k)}{|}_{\infty}^m\right)\leq D_k^{m-1}\sum_{n=1}^{\infty}d_{k,n}\E\left(Z_n^m\right)=
D_k^{m}
\E\left(Z_1^m\right)<\infty,
\]
it concludes that condition (H2) holds true for $k$ and $m$.
\end{proof}

\begin{remark}
Given a kernel 
$g^{(k)}\colon\R\to\R$ 
we define by 
$G^{(k)}\colon[0,\infty)\to [0,\infty)$ 
the monotone hull of 
$\max(|g^{(k)}(t)|,|g^{(k)}(-t)|),\ t\geq 0$, 
as the smallest non-increasing function that dominates 
$|g^{(k)}(t)|$ 
and 
$|g^{(k)}(-t)|$ 
for 
$t\geq 0$.
More precisely  $G^{(k)}(t) = \sup_{s>t} (|g^{(k)}(s)|,|g^{(k)}(-s)|)$.
Then, condition \eqref{eq:dk} is implied by the integrability of $G^{(k)}$.
Then, when the kernel $g^{(k)}$ decreases monotonously for large positive and large negative values, condition \eqref{eq:dk}
follows automatically from the integrability of the kernel $g^{(k)}$.
\end{remark}

\noindent{\bf Boundedness of the density  {\color{black}  (H4)}.} 
In \cite{bd1} Section 3.2, it is shown that when $\beta_1=1$ a.s., in the two following particular situations
the stationary shot noise process has a bounded density:
\begin{itemize}
\item[(a)] The kernel is $g(t)=e^{-t}\indicator_{\{t\geq 0\}}$, and the intensity $\lambda>1$.
\item[(b)] The kernel satisfies $g(t)=t^{-\alpha}$ for $t\geq A$ for some $A>0$ and $\alpha>1/2$.
\end{itemize}

We present below a generalization of the results of \cite{bd} that constitutes one of the contributions
of the present paper.
\begin{proposition}\label{proposition:B} 
Consider a shot noise process \eqref{eq:sn} with a differentiable kernel.
Denote by $T$ an exponential random variable with parameter $\lambda$.
Assume that either
\begin{itemize}
\item[\rm(A)] The impulse $\beta_1$ has a density bounded by $B$ and  either $\E(1/|g(-T)|)<\infty$ or $\E(1/|g(T)|)<\infty$.
\item[\rm(B1)]  Assume that $\E(1/|\beta_1|)<\infty$.
Assume further that $g(x)$ has a strictly negative derivative in $(0, +\infty)$,
and the function  
\[
g^+_*(t):=\inf_{0\leq s\leq t}|g'(s)|,\; (t\geq 0),
\] 
satisfies $\E(1/g^+_*(T))<\infty$.
\item[\rm(B2)]   The same as  (B1) replacing $g(x)$ by  $g(-x)$ and $g^+_*(t)$ by 
\[
g^-_*(t):=\inf_{0\leq s\leq t}|g'(-s)|,\; (t\geq 0),
\]
\end{itemize}
Then $X(0)$ has a bounded density.
\end{proposition}

Note that Proposition \ref{proposition:B} implies {{\color{black}(H4)} with $h=0$. 
The proof requires the following simple result, that has a direct proof.
\begin{lemma}\label{lemma:simple}
Consider two independent random variables $X$ and $Y$, 
where $X$ has a density bounded by $B$. 
\par\noindent{\rm(a)}
Then, the sum $X+Y$ has a density bounded by $B$.
\par\noindent{\rm(b)}
If $\E(1/|Y|)<\infty$, the product $XY$ has a density bounded by $B\E(1/|Y|)$.
\end{lemma}

\begin{proof}[Proof of Proposition \ref{proposition:B}]
Consider (A), assuming that $\E(1/|g(-T)|)<\infty$ (the other case is analogous).
Based on Lemma \ref{lemma:simple} and the decomposition
\begin{equation*}
X(0)=\sum_{i\colon\tau_i>0}\beta_ig(-\tau_i)+\sum_{i\colon\tau_i\leq 0}\beta_ig(-\tau_i)=:X^+(0)+X^-(0),
\end{equation*}
as $X^+(0)$ and $X^-(0)$ are independent, it is enough to see that $X^+(0)$ has a bounded density.
Define by $T_1$ and $T_2$ the first two positive occurrences of the Poisson process.
Condition on $T_2$ and apply (b) in Lemma \ref{lemma:simple} to obtain that $\beta_1g(-T_1)$ has a conditional density
bounded by $B\E(1/|g(-T_1)|\mid T_2)$.
Applying now (a) in Lemma \ref{lemma:simple}, we obtain that the sum $X^+(0)$ has a conditional
density, denote it by $f_{X^+(0)\mid T_2}(x)$ with the same bound. Finally, integrating
\[
f_{X^+(0)}(x)=\E(f_{X^+(0)\mid T_2}(x))\leq \E(B\E(1/|g(-T_1)|\mid T_2))=B\E(1/|g(-T_1)|),
\]
concluding the proof in this case.

Let us consider now case {\color {black} (B1), the proof in the case (B2) is similar}.
We see first that, conditional on $ T_2$, the random variable
 $g({\color{black} -T_1})$ has a bounded density.
In fact, by the change-of-variable formula
\[
  f_{g(-T_1)\mid T_2}(u)=
  \frac1{|g'(g^{-1}(u))|}\frac1{T_2}\indicator_{\{g^{-1}(u) \leq T_2\}}\leq {1\over T_2g^+_*(T_2)}.
\]
By (b) in Lemma \ref{lemma:simple} the product $\beta_1g(T_1)$ has a conditional density bounded by 
\begin{equation}\label{bound}
{1\over T_2g^+_*(T_2)}\E({1/|\beta_1|}).
\end{equation}
Then, by conditional independence and (a) in Lemma \ref{lemma:simple}, the sum $X^+(0)$ 
has a conditional density with the same bound.
Finally, as the density of $T_2$ is $\lambda^2t_2e^{-\lambda t_2}$,
integrating the bound \eqref{bound}, we obtain that
\[
f_{X^+(0)}(u)=\E(f_{X^+(0)\mid T_2}(u))\leq \int_0^{\infty}{1\over t_2g_*^+(t_2)}\E({1/|\beta_1|})\lambda^2t_2e^{-\lambda t_2}dt_2=\lambda 
\E({1/|\beta_1|})\E({1/g^+_*(T)}).
\]
This concludes the proof of the proposition giving  the respective bounds.
\end{proof}

%%%
\begin{corollary}
Consider a shot noise process with a $C^\infty(\R)$ kernel $g$ such that for every $k\geq 1$  
\begin{equation}\label{exp}
g ^{(k)}(t) \simeq  c_k(t) e^{\alpha t}  \text{ as  } t\to -\infty 
\quad\text{and }\quad 
g^{(k)}(t) \simeq C_k(t)e^{-\alpha t} 
\text{ as  } t\to +\infty,
\end{equation}
where $\simeq$  means equivalence  and $c_k$, $C_k$ are polynomials of an arbitrary degree. 
Assume furthermore that   $\lambda>\alpha$ and  $\E(|\beta_1|)<\infty$. 
If in addition   either 
\begin{itemize}
\item[\rm(i)] $\beta_1$ has bounded density or
\item[\rm(ii)] $\E(1/|\beta_1|)<\infty$ and  $g$, or $g(-x)$,  has a strictly negative  derivative on $(0,+\infty)$,
\end{itemize}
then 
\[
\E(N_u^p)<\infty,\quad\text{for any $p\geq 1$.}
\]
\end{corollary}
\begin{proof}  We apply Corollary \ref{thm:shotnoiseR} with $m=1$ and arbitrarily large $k$. 
In view of \eqref{eq:pnw}, we derive the result for arbitrary $p$.
The kernel is differentiable for any $k$. The relation \eqref{exp} ensures condition \eqref{eq:dk},
as it gives the integrability of the derivatives of any order of the kernel.
It remains to see that the density of $X(0)$ is bounded, and this follows in case (i) from the
fact that $\lambda>\alpha$ giving $\E (1/|g(-T_1)|)<\infty$, and the boundedness of the density follows by 
(A) in Proposition \ref{proposition:B}.
In case of (ii), we apply (B) in Proposition \ref{proposition:B}.
In this way we conclude the proof.
\end{proof}

%%%%%%%%%%%%%%%%%%%%%% RANDOM FIELDS
\section{Random fields from $\R^d$ and $S^d$ to $\R$}\label{section:co-one}
We begin with the case when the domain is $\R^d$. 
Consider a real valued random field $\X=\{X(t)\colon t\in \R^d\}$ and
define, for a given $u\in\R$,
the level set $C_u$ restricted to $\D_a$ (the closed  ball with radius $a$ centered at the origin) 
by the formula
\[
C_u=C_u(\D_a):=\{t\in \D_a\colon X(t)=u\}.
\]
Observe that under regularity conditions, the level set $C_u$ is almost surely a manifold of co-dimension one. 
{\color{black} 
In fact, Theorem 1 in \cite{belyaev} (see \cite{w-ln} for a proof) 
states that $\{t\in \D_a\colon X(t)=u, \nabla X(t)\textrm{ is singular}\}$ 
is a.s. empty and, provided the a.s. absence of critical points, 
the Implicit Function Theorem gives a local chart a.s. 
}
% (see 
% \cite{aw}).  
The aim of the first part of this section is to generalize Theorem \ref{theorem:key} into this framework.

To this end, we compute the $(d-1)$-dimensional Hausdorff measure  of a Borel set $B{\color{black}\subset \D_a}$, based on
Crofton's formula (\cite{Morgan} p.  31): \begin{equation}\label{eq:igm}
  \mathcal{H}_{d-1}(B)=c_{d-1}(a)\int_{v\in S^{d-1}}\int_{y\in
v^\perp\cap \D_a} \# \left\{ B\cap \ell_{v,y} \right\} \,dv_a^\perp(y)\,
dS^{d-1}(v).  \end{equation}
Here $dS^{d-1}$ is the uniform probability on the sphere
$S^{d-1}$,  $dv_a^\perp$ is the uniform probability on $v^\perp\cap \D_a$,
and $\ell_{v,y}$ is the affine linear space 
$\{y+tv:\, t\in\R\}$. 
The constant $c_{d-1}(a)$ can be easily computed in the particular case
of the boundary of $\D_a$, 
namely 
$S^{d-1}_a:=\{t\in \R^{d}:\,{|}t{|}=a\}$,  
yielding, 
\begin{equation}\label{eq:cd}
c_{d-1}(a)= \frac{1}{2}\mathcal{H}_{d-1}(S^{d-1}_a)=
\frac{\pi^{d/2}}{\Gamma(d/2)}a^{d-1}.
\end{equation}
\begin{remark}
  When $B\subset \D_a$ is a codimension one smooth submanifold
  $\mathcal{H}_{d-1}(B)$ coincides
  with the (induced) Riemannian measure. 
\end{remark}

In view of \eqref{eq:igm}, to obtain the finiteness of the moments of
$\mathcal{H}_{d-1}(C_u)$, the idea is to give conditions on $\X$ that
ensure that its restriction to an arbitrary line in $\R^d$ 
verifies the hypothesis of Theorem \ref{theorem:key}
for some values of $m,h$ and $k$, as stated in the following result.

\begin{theorem} \label{t:codim}     
Consider a real valued random field $\X=\{X(t)\colon t\in \D_a\}$. 
Assume that $\X$ satisfies 
\begin{itemize}
\item[\rm(H1${}^\prime$)] The sample paths of $\X$ are $C^k(\D_a)$, {\color{black} $k \geq 2$}.
\item[\rm(H2${}^\prime$)] For some $m=1,2,\dots$ there exists a constant $D_m$ such that
\begin{equation*}%\label{eq:sup}
\max_{{|}v{|}=1}\E\left(\left|{\partial^kX\over\partial v^k}\right|_{\infty}^m\right)\leq D_m,
\end{equation*}
where ${\partial^k \over\partial v^k}$ denotes the 
$k$-th directional derivative w.r.t. the vector $v$.
\item[\rm(H3${}^\prime$)] For some $0\leq h\leq k$ there exists a constant $C>0$ such that
the joint density of 
$$
X(t),{\partial X(t)\over\partial v},\dots,{\partial^hX(t)\over\partial v^h}
$$ 
is bounded by $C$ uniformly in $t\in \D_a$, 
$v\in S^{d-1}$, and   {\color{black}  $(u_1,\ldots, u_{h+1 }) \in \R^{h+1}$}. 
\end{itemize}
 Then, for $p\geq 1$  satisfying  \eqref{eq:nw},
the $p$-th moment of  the measure of the level set
  $\mathcal{H}_{d-1}(C_u)$ is finite and is bounded by
$$
     (c_{d-1}(a))^p
   \,\left((k-1)^p+D_m\cdot E_{\alpha,k,p}+C|2a|^{(h+1)(k-h/2)}\cdot
   D_{\alpha,k,h,p}\right),
$$
where $\alpha$ and the coefficients  $E_{\alpha,k,p}$ and $D_{\alpha,k,h,p}$ are as in Theorem \ref{theorem:key}.
\end{theorem}

\begin{proof}[Proof of Theorem \ref{t:codim}]
We first apply Jensen's inequality in \eqref{eq:igm}:
\begin{align*}
\left( \mathcal{H}_{d-1}(C_u)\right)^p&=
  \left(
c_{d-1}(a)\int_{v\in
  S^{d-1}}\int_{y\in v^\perp\cap \D_a} \#
  \left\{ C_u\cap \ell_{v,y} \right\} \, dS^{d-1}(v)\,dv_a^\perp(y)
 \right)^p\\
  &\leq (c_{d-1}(a))^p  
  \int_{v\in S^{d-1}}\int_{y\in v^\perp\cap \D_a} 
  \left(\#\left\{ C_u\cap \ell_{v,y} \right\}\right)^p \,
  dS^{d-1}(v)\,dv^\perp_a(y).
\end{align*}
Now, take expectation and apply Tonelli's Theorem, 
\begin{equation*}
\E\left( \mathcal{H}_{d-1}(C_u)\right)^p
  \leq (c_{d-1}(a))^p 
  \int_{v\in S^{d-1}}\int_{y\in v^\perp\cap \D_a} 
 \E \left(\#\left\{ C_u\cap \ell_{v,y} \right\}\right)^p \,
  dS^{d-1}(v)\,dv^\perp_a(y).
\end{equation*}
Let us apply Theorem \ref{theorem:key} to the expectation inside the integral. 
It is clear that the bound is maximal when the interval is maximal, and it corresponds to $y=0$.
So,
 \begin{equation*}
\E\left( \mathcal{H}_{d-1}(C_u)\right)^p
   \leq (c_{d-1}(a))^p
   \,\left((k-1)^p+D_m\cdot E_{\alpha,k,p}+C|2 a|^{(h+1)(k-h/2)}\cdot
   D_{\alpha,k,h,p}\right).
\end{equation*}
The expectation in the r.h.s. above is finite due to Theorem \ref{theorem:key}.
\end{proof}
{\color{black}  An easy consequence is }
\begin{corollary}
Assume that the random field in Theorem \ref{t:codim}  is {\color{black} stationary }Gaussian
with a spectral density.
Then, the moment of order $p$ of the Hausdorff measure of $C_u$  is finite
with $p= {k(k+1)\over 2}-2$.
\end{corollary}
{\color{black} \begin{proof} The only point to check is that, if we limit the parameter set to a line, the obtained Gaussian stationary process 
has a density so it satisfy the conditions of Example \ref{ex:stat}. \end{proof}}
%%%%%%%%%%%%%%%%%%%%%%%%%%%%%%%%%%%%%%%
Let us mention that the case  $p=2$ and the case of arbitrary $p$ were considered in \cite{zazachichi} and  in \cite{AAGL}
respectively.\\

% SPHERE

We move to random fields defined on a sphere. 
Without loss of generality, we can assume that the sphere  is 
$S^d=\{x\in\R^{d+1}\colon
|x|=1\}$. 

Given $t\in S^d$  and a tangent vector $v\in T_tS^d$, we denote by
$\partial^kX(t)\over\partial v^k$ the $k$-th order derivative of $X$
along the sphere.

%%%%%%%%%%%%%% Crofton
For  the definition of the integral geometric measure on
  homogeneous spaces, we refer \cite{Santalo} or \cite{Howard}.
In this case, if $\M\subset S^d$ is a co-dimension one regular set, the Crofton's formula reads
\begin{equation}\label{for:croftonS}
  \mathcal{H}_{d-1}(\M)=\beta_{2,d+1}  
  \int_{E\in G_{2,d+1}}
  \#(\M\cap
  E)\,dG_{2,d+1}(E).
  \end{equation}
  Here,  $G_{2,d+1}$ denotes the Grassmanian of $2$-dimensional subspaces of
  $\R^{d+1}$, and the integral is with respect  to the induced Haar
  measure as homogeneous space of the orthogonal group of $\R^{d+1}$. This is the unique probability measure that is invariant under the
  action of this group. 
\begin{remark}
  It is easy to see that the given probability measure on $G_{2,d+1}$
    can be generated by the span of two independent standard Gaussian
    vectors on $\R^{d+1}$.
\end{remark}

\begin{remark}\label{rem:greatcir}
  Note that $E\cap S^d$ is a great circle for every $E\in G_{2,d+1}$.
\end{remark}
 
  The constant $\beta_{2,d+1}$ can be computed in the same fashion as
  in (\ref{eq:cd}). In this case, we have  
  $$
  \beta_{2,d+1} = \frac{\pi^{d/2}}{\Gamma(d/2)}.
  $$

%%%%%%%%%%
\begin{theorem} \label{t:S}     
Consider a real valued random field $\X=\{X(t)\colon t\in S^{d}\}$. Assume that: 
\begin{itemize}
\item[\rm(H1${}^{\prime\prime}$)] $\X$ is $C^k(S^d)$ for a given $k$.
\item[\rm(H2${}^{\prime\prime}$)] there exist $m\in\N$ and a constant $D_m$ such that
\begin{equation*}%\label{eq:sup}
 \E\left( \max_{t\in S^d, v\in T_tS^d,{|}v{|}=1}\left|{\partial^kX(t)\over\partial v^k}\right|^m\right)\leq D_m.
\end{equation*}
\item[\rm(H3${}^{\prime\prime}$)] Let $0\leq h\leq k$. Assume that there exists a constant $C>0$ such that
the joint density of 
$$
X(t),{\partial X(t)\over\partial v},\dots,{\partial^hX(t)\over\partial v^h}
$$ 
is bounded by $C$ uniformly in $t\in S^d$ and $v\in T_tS^d,\,{|}v{|}=1$, and for  {\color{black} 
$(u_1,\ldots, u_{h+1 }) \in \R^{h+1}$}. \\
 \end{itemize}
  Then, for $p=1,2,\dots$, satisfying  \eqref{eq:nw}
the $p$-th moment of  the measure of the level curve
  $\mathcal{H}_{d-1}(C_u)$ is finite and bounded by
$$
    \beta_{2,d+1}^p   \,\left((k-1)^p+D_m\cdot E_{\alpha,k,p}+C|2\pi|^{(h+1)(k-h/2)}\cdot
   D_{\alpha,k,h,p}\right),
   $$
where  $\alpha$ and the coefficients
$E_{\alpha,k,p}, D_{\alpha,k,h,p}$ are given in
Theorem \ref{theorem:key}.
\end{theorem}
\begin{proof}
  From Remark \ref{rem:greatcir} and (\ref{for:croftonS}) applied to $\M=C_u$, we observe that
$$
\#(C_u\cap E)=\#\{t\in E\cap S^d\colon X(t)=u\}$$
is the number of crossings of a process defined on a  great circle of
  radius $1$ in such a way that Theorem \ref{theorem:key} can be
  applied. Now, the proof follows the same lines as proof of Theorem
  \ref{t:codim}.
\end{proof}

\section{Shot noise random field}\label{snp}
In this section, we first consider a shot noise defined in $\R^d$ 
and then on the sphere $S^d$.

Consider a shot noise random field $X\colon\R^d \to \R$ defined by \eqref{eq:sn},
where we now consider  the case $d>1$.
We are interested in the finiteness of high-order moments of the measure of level sets of this differentiable random field which is
restricted to the closed centered ball $\D_a$. For the computation of expectations of excursion sets of shot noise random fields with realizations with bounded variation 
(possibly discontinuous), see \cite{bd2}.

The following result is a direct application of Theorem \ref{t:codim}. 
\begin{corollary}\label{theorem:4}
Consider a shot noise random field \eqref{eq:sn} satisfying condition 
{\rm(H1${}^\prime$)} for some $k\geq 1$,
{\rm(H2${}^\prime$)} for some $m\geq 1$ and $k$ above, and 
{\rm(H3${}^\prime$)} for $h=0$.
Then,
\begin{equation}\label{eq:hausdorff}
\E (\mathcal{H}_{d-1}(C_u(\D_a))^p)<\infty
\end{equation}
for 
\begin{equation*}
p<{m\over m+1}(k-1).
\end{equation*}
\end{corollary}
In the following paragraphs, we obtain sufficient conditions to verify this corollary.
The differentiability of the trajectories follows directly from the differentiability of the kernel $g${\color {black}, see \cite{bd2}}.
Consider a multi-index $\alpha=(\alpha_1,\dots,\alpha_d)$ of non-negative integers such that 
$|\alpha|=\sum_i\alpha_i=k$ and the usual notation for the partial derivatives. Then,
\[
{\partial^{k} X(t)\over \partial t^\alpha}=\sum_i\beta_i{\partial^{k} g(t)\over \partial t^\alpha}(t-\tau_i),
\]
is a shot noise provided that the $\alpha$-th
derivative of $g$ is integrable.\\

\noindent{\bf Boundedness of moments (H2$^\prime$).} 
For simplicity of exposition, we take $a=1$, i.e. the unitary ball $\D_1$.
Define the partition of $\R^d$ given by $A_1=\D_1$ and
\[
\{A_n=\D_{r_n}\setminus \D_{r_{n-1}}\colon n=2,3,\dots\},
\]
where the sequence $(r_n)$ has $r_1=1$ and is such that the volume of each set $A_n$ equals the volume of $A_1$. 
In this way, we can write
\[
\frac{\partial^{k} X}{\partial v^k}(t)
=\sum_{n=1}^\infty\sum_{i:\tau_i\in A_n}\beta_i \frac{\partial^{k} g}{\partial v^k}(t-\tau_i).
\]
We now present a result useful to verify the moment condition (H2$^\prime$). 
The proof follows the same lines as that of Proposition \ref{p:dkn} in dimension one, and is omitted.
\begin{proposition}\label{prop:snimpulse}
Consider a shot noise with impulse s.t. $\E(|\beta_1|^m)<\infty$ for given $m$, and 
kernel s.t. 
\begin{equation*}%\label{eq:dkd}
D'_{k}=\sum_{n=1}^\infty d'_{k,n}<\infty,\text{ for given $k$},
\end{equation*}
with 
\[
d'_{k,n}=\sup_{{|}v{|}=1}\sup_{t\in A_1,s\in A_n}\left|\frac{\partial^kg}{\partial v^k}(t-s)\right|.
\]
Then, the condition {\rm(H2$^\prime$)} holds true for $m$ and $k$ as above, i.e.
\[
\sup_{{|}v{|}=1}\E\left(\left|\frac{\partial^kX}{\partial v^k}\right|_{\infty}^m\right)<\infty.
\]
\end{proposition}

\noindent{\bf Boundedness of the density  (H3$^\prime$).} 
Introduce the volume of the $d$-dimensional ball of radius $r$ by
\begin{equation*}%\label{eq:volumen}
V(r)=k_dr^d,\quad k_d=\pi^{d/2}/\Gamma(1+d/2).
\end{equation*}
\begin{proposition}\label{proposition:3} 
Consider a shot noise random field \eqref{eq:sn} with a differentiable kernel. 
Denote by $T_1$ the occurrence of the Poisson field closest to the origin with the Euclidean distance (a.s. defined).
Assume that either
\begin{itemize}
  \item[\rm(A)] The impulse $\beta_1$ has a density bounded by $B$ and $\E(1/|g(-T_1)|)<\infty$.
\item[\rm(B)]  $\E(1/|\beta_1|)<\infty$ and the function 
\[
G_*(r):= \sup_{u\in\R}\int_{t\in g^{-1}(u) \cap B_r }
\frac1{\left|\nabla
  g(t)\right|}d\mathcal{H}_{d-1}(t)
\] 
  satisfies 
  $\E (G_*(|T_1|))<\infty$.
\end{itemize}
Then $X(0)$ has a bounded density.
\end{proposition}
The following auxiliary result is needed in the proof under (B). 
Its proof is elementary and  thus omitted. 
\begin{lemma}\label{densities}
 Denote by $(T_n)$ the occurrence of the Poisson field
  ordered by the Euclidean distance to the origin.
Then, the random variables $|T_1|$ and $|T_2|$ have the following densities:
  \begin{equation*}%\label{densities}
  f_{|T_1|}(x)=d\lambda k_de^{-\lambda k_d x^d} x^{d-1},
 \quad
  f_{|T_2|}(x)=d\lambda^2 k_d^2e^{-\lambda k_d x^d} x^{2d-1}.
  \end{equation*}
\end{lemma}
% \begin{proof}
% Denote by $N(A)$ the number of occurrences of the Poisson field in the set $A$,
% and $A(a,b)=B_b\setminus \D_a$.
% We have
% \begin{align*}
% \P(x\leq |T_1|\leq x+h)	&\sim \P(N(B_x)=0)\P(N(A(x,x+h))=1)
% \sim e^{-\lambda k_d x^d}\lambda dk_dx^{d-1}h,
% \end{align*}
% proving the first formula. The second is similar:
% \begin{align*}
% \P(x\leq |T_2|\leq x+h)	&\sim \P(N(B_x)=1)\P(N(A(x,x+h))=1)%\\
% \sim e^{-\lambda k_d x^d}\lambda^2 dk_d^2x^{2d-1}h,
% \end{align*}
% concluding the proof.
% \end{proof}

\begin{proof}[Proof of Proposition \ref{proposition:3}]
This proof follows the same lines  as those of the proof of
Proposition \ref{proposition:B}.
To prove case (A),
we write
\begin{equation*}
X(0)=\beta_1g(-T_1)+\sum_{i=2}^{\infty}\beta_ig(-T_i).
\end{equation*}
%We begin by showing that $\beta_1g(-T_1)$ has a bounded density conditional on $T_2$.
Applying (b) in Lemma \ref{lemma:simple}, we obtain that $\beta_1g(-T_1)$ has a conditional density
bounded by $B\E(1/|g(-T_1)|\mid T_2)$.
Applying now (a) in Lemma \ref{lemma:simple}, by conditional
independence we obtain that the sum $X(0)$ has a conditional
density: $f_{X(0)\mid T_2}(x)$ with the same bound. Finally, integrating
$$
f_{X(0)}(x)=\E(f_{X(0)\mid T_2}(x))\leq \E(B\E(1/|g(-T_1)|\mid T_2))=B\E(1/|g(-T_1)|),
$$
concluding the proof in this case.

Let us consider now case (B).
Applying the co-area formula we see that, conditional on $T_2$, the random variable
$g(-T_1)$ has a bounded density:
$$
f_{g(-T_1)\mid T_2}(u)
=\int_{t\in g^{-1}(u) \cap \D_{|T_2|} }
\frac1{|\nabla g(t)|}\frac1{k_d|T_2|^d}
  \mathcal{H}_{d-1}(dt)
\leq { G_*(|T_2|)\over k_d|T_2|^d}.
$$
By (b) in Lemma \ref{lemma:simple}, the product $\beta_1g(-T_1)$ has a conditional density bounded by 
\begin{equation}\label{bound2}
{ G_*(|T_2|)\over k_d|T_2|^d}\E({1/|\beta_1|}).
\end{equation}
Then, by conditional independence and (a) in Lemma \ref{lemma:simple}, the sum $X(0)$ 
has a conditional density with the same bound \eqref{bound2}. 
Integrating the bound \eqref{bound2} with the density of $|T_2|$ in Lemma \ref{densities},
we obtain that
$$
\aligned
f_{X(0)}(x)=\E(f_{X(0)\mid T_2}(x))
&\leq 
{\E({1/|\beta_1|})\over k_d}
\int_0^\infty \frac{G_*(r)}{r^d}
e^{-\lambda k_d r^d}\lambda^2 dk_d^2 r^{2d-1}
dr\\
&=   
dk_d\lambda^2\E({1/|\beta_1|})
\int_0^\infty G_*(r)
e^{-\lambda k_d r^d} r^{d-1}dr=
\lambda\E({1/|\beta_1|})\E(G_*(|T_1|)),
\endaligned
$$
which is finite because of our hypotheses, concluding the proof.
\end{proof}
\begin{example}
Consider a shot noise with a radial kernel of the form $g(t)=e^{-|t|^{2q}}$ ($q=1,2,\dots$),
and  impulse s.t. $\E(|\beta_1|)<\infty$.
Assume further that
\begin{itemize}
\item[(i)] $2q<d$, or $2q=d$ and  $\lambda k_d>1$,
\item[(ii)] $\beta_1$ has a bounded density or  $\E(1/|\beta_1|)<\infty$. 
\end{itemize}
Let us check the finiteness of the moments \eqref{eq:hausdorff}.
It is straightforward to verify  (H1${}^\prime$) and (H2${}^\prime$).
To see (H3${}^\prime$) assume first that $\beta_1$ has a bounded density.
Then, as $g$ is radial, we have
\begin{equation}\label{eq:integral}
\E\left(\frac1{g(-T_1)}\right)=\E\left(\frac1{g(|T_1|)}\right)=\int_0^\infty \frac1{g(r)}f_{|T_1|}(r)dr<\infty,
\end{equation}
due to the form of the kernel, Lemma \ref{densities}, and condition (i).

Assume now that $\E(1/|\beta_1|)<\infty$. 
The sup in the function $G_*(r)$ in (B) in Proposition \ref{proposition:3}
is  attained when $u=e^{-r^{2q}}$, giving
$
G_*(r)=e^{r^{2q}}r^{d-2q}.
$
A computation similar to \eqref{eq:integral} gives $\E G_*(|T_1|)<\infty$.
In conclusion, 
the hypotheses of Corollary \ref{theorem:4} are valid with arbitrary $k$, 
obtaining that the shot noise random field \eqref{eq:sn} verifies \eqref{eq:hausdorff} with arbitrary $p$.
\end{example}

Now, we move to the shot noise random field defined on the sphere. 
Let  $\mathcal  P$ be the standard Poisson  field on the unit sphere
$S^{d}$ of $\R^{d+1}$.
%Fix $x_0\in S^{d}$. Then, almost surely, we can define by $(T_n)$ the
%occurrence of $\mathcal P$ ordered by the geodesic distance to $x_0$. 
%ote that, conditional on the number of points of
We realize $\mathcal P$ as a sequence $T_1,T_2,\dots,T_N$ of uniformly distributed random points on $S^{d}$ where $N$ is a Poisson random variable, 
all variables being  independent.

Let $g \  : [0,\pi) \to \R$ be a $\mathcal C^{\infty}$  function. 
The shot noise process is defined by
 \begin{equation} \label{e:sn}
   X(t) = \sum_{i=1}^N \beta_i g( \mbox{dist}^2(t,T_i)),\quad t \in
   S^d,
\end{equation}
where $dist$ is the geodesic distance, and the $\beta_i $'s are i.i.d.
with distribution  $F$. 

We assume the following conditions:
\begin{itemize}
  \item[\rm(H4)] $\beta_1 g(\mbox{dist}^2(t,T))$
    admits a bounded density, where $T$ has a uniform distribution on the sphere. 
  \item[\rm(H5)] The  distribution $F$  has moments of any order.
%\item[\rm(H6)] 
%The function $g$  as well as its derivatives of any order,  are bounded.
\end{itemize}
Note that (H4) is rather  weak, it is met, for example, if  $\beta_1$ has a bounded density  
and  $g$ is strictly positive on $[0,\pi]$ (apply Lemma
\ref{lemma:simple}). 
We have the following result.
\begin{proposition} \label{proposition:sn}
Consider a shot noise random field defined by
\eqref{e:sn} and assume (H4)-(H5). 
Then, 
for every level $u\neq 0$, for every compact set $W$  and for every integer $p$
$$
 \E\left( \mathcal{H}^p_{d-1}(C_{u}\cap W)\right) <\infty.
$$
\end{proposition}
The proof follows from Theorem \ref{t:S}. \\

\noindent{\bf Differentiability (H1${}^{\prime\prime}$) and Boundedness of moments (H2${}^{\prime\prime}$).}
We now check (H1${}^{\prime\prime}$) and (H2${}^{\prime\prime}$) for arbitrary  $k$ and $m$.

Let $v$ be a norm $1$ vector orthogonal to  $t\in S^{d}$.
Since the number of realizations of $\mathcal P $ is almost surely finite, 
for any $k$, the derivative of $X$ along the sphere at $t\in S^{d}$ in the direction of $v$ is
given by
\begin{equation*} %\label{e:snp} 
  \frac{\partial ^k X(t)}{\partial v^k} = \sum_{i=1}^N \beta_i
   \frac{d^k}{dz^k}g(\mbox{dist}^2(\gamma_{t,v}(z),T_i))\Big|_{z=0}, 
\end{equation*}
where $\gamma_{t,v}(z)$ is a $\mathcal{C}^\infty$ arc-length parametrization of a
great circle through $t$ in the direction of $v$ such that
$\gamma_{t,v}(0)=t$. 

Then, $X$ is $\mathcal C^k$ for all $k$, and by compactness of the sphere $S^{d}$ we obtain
$$
  \frac{\partial ^k X(t)}{\partial v^k} \leq
    (Const) \sum_{i=1}^N  \beta_i.
$$
The condition (H5) implies that the compound Poisson distribution of $\sum_{i=1}^N \beta_i$ admits moments of every order giving the desired 
result.\\

\noindent{\bf Boundedness of the density (H3${}^{\prime\prime}$) with $h=0$.} 
For simplicity, we consider the hypothesis (H3$^{\prime\prime}$) with $h=0$ and study the  marginal density of $X(t)$.
\begin{lemma} \label{l:sn}
For every $t \in S^{d}$ the distribution of $X(t)$ is the sum of 
% \begin{itemize}
one atom at zero and 
a defective probability with  bounded density.
% \end{itemize}
\end{lemma}
\begin{proof}
We consider the distribution of $X(t)$ conditional to $\mathcal P (S^{d}) = k$ in the case $k>0$. 
Under that condition it is well known  that the $T_i$'s, $i=1\ldots,k$, 
are i.i.d. with uniform distribution on $S^{d}$, so we write 
$$
   X(t) = \beta_1 g(\mbox{dist}^2(t,T_1)) +  \sum_{i=2}^k \beta_i g(\mbox{dist}^2(t,T_i)).
$$
The terms in this sum are independent  and, because of (H4), 
the first term has a bounded density. 
By convolution it is the same for the conditional distribution 
of $X(t)$ which admits a density bounded by the same constant. 
Since  this bounds  does not depend on $k$, it is also a bound for the density  of $X(t)$ 
conditional to 
$\mathcal P(S^d) >0$. 
Obviously,  when $\mathcal P(S^{d})=0$, $X(t)=0$ which gives the atom  at zero.
\end{proof}

\section{Non Gaussian Kac-Rice Formula}\label{section:jma}
Gaussian  KRFs are valid  under weak and simple conditions and a comprehensive reference is the book \cite{aw} that treats all the 
relevant dimensions, i.e. random fields $X$ from $\R^D$ to $ \R^{d}$  with $d \leq D$. 
These formulas give the expectation or the higher moments  of the
$\mathcal{H}_{D-d}$ Hausdorff measure of 
$$
C_u(H)=\{t\in H\colon X(t)=u\},
$$
that is the level set restricted to a compact set $H\subset\R^D$.
Though the proofs use basically the change-of-variable formula (or its generalization: the co-area formula) and have  nothing to do with 
Gaussianity, its generalization to non-Gaussian  cases  encounters difficulties 
in defining properly the quantities involved in the formulas.

For instance, in the simplest case, the KRF for the expectation when $D=d=1$ formally reads for a compact interval $I$:
\begin{equation}\label{e:jma:rice}
 \E(N_u(I))  = \int_I \E\big(|X'(t)| \big| X(t) =u \big) p_{X(t)} (u) dt.
\end{equation}
In the non-Gaussian case, the conditional expectation is defined only for almost every level $u$. 
As a consequence, the  punctual values of the r.h.s.  of  \eqref{e:jma:rice} are not defined 
unless some kind of continuity is established. 
This is why  the non-Gaussian KRF requires complicated conditions. 
See \cite{marcus} and \cite{aw} for the case $D=d=1$; 
for the case $D=d>1$ the only reference is \cite{adler}. 
To our knowledge, in the non-Gaussian case, with the exception of the complicated treatment in \cite{w-ln}, 
there exists no proof of the KRF for the case $D>d$. %***Mario***

Often,  the process  or the random field  has $C^k $ paths with $k$ ``large" and the conditions can be drastically simplified. This is the object of 
this section. Note  that in its full generality,  the statement of the KRF cannot be stated in the classical form.\\

%{\bf Results}
Our first main result is the following: 
\begin{theorem} \label{t:jma:1}
Let us consider  a  real-valued process $ \X =\{ X(t), t \in I \}$, where $I$ is a bounded interval of $\R$. 
Assume that there exist $u\in\R$ and $\epsilon>0 $ such that:
\begin{itemize}
\item[\rm(a)] The sample paths of $\X$ are $C^1(I)$.
\item[\rm(b)] 
$
\sup_{|v-u|<\epsilon}\E(N_v(I)^2)<\infty.
$ 
\item[\rm(c)] The density of $X(t)$ at point $v$ 
is uniformly bounded for $t\in I$ and for  $|v-u|<\epsilon$.
\end{itemize}
Then, 
 \begin{equation}\label{e:ricegen}
 \E(N_u(I))  = \lim_{\delta \to 0} \frac{1}{2\delta} \int_I \E\big(|X'(t)| \indicator_{|X(t)-u| \leq \delta} \big) dt.
\end{equation}
Suppose in addition that $X'(t)$  has a finite expectation for every $t \in I$. 
Define for $|v-u|<\epsilon$
$$
R(v) := \int_I \E\big(|X'(t)| \big| X(t) =v \big) p_{X(t)} (v) dt.
$$
Then, \eqref{e:ricegen} can be rewritten as 
\begin{equation} \label{a:jma:2}
\E(N_u(I))  = \lim_{\delta \to 0} \frac{1}{2\delta} \int_{u-\delta}^{u+\delta} R(v) dv.
\end{equation}
\end{theorem}
\begin{proof} 
By (b) there exists a finite $K$ such that 
$$
\E(N_v^2)  \leq K,
$$
for all $|v-u|<\epsilon$.
Our conditions imply that, with probability 1, the process 
$X(t)$  cannot take the value $u$ at the two extremities of $I$. In addition by the Bulinskaya  Lemma (Prop. 1.20 in  \cite{aw})  there are, with 
probability 1, no extremes  at the level $u$. 
Thus, the Kac Lemma (Lemma  3.1 in  \cite{aw}) yields
$$
N_u = \lim_{\delta \to 0} N_u^\delta , \ \mbox{almost surely},
$$
where $N_u^\delta$ is the Kac's counter defined by
$$
 N_u^\delta := \frac{1}{2\delta} \int_I |X'(t) | \indicator_{\{|X(t)-u| \leq \delta\}}  dt.
$$
By the area formula (Prop 6.1 in  \cite{aw})  
$$
 N_u^\delta = \frac{1}{2\delta} \int_{u-\delta}^{u+\delta}   N_v dv.
$$
This, associated to the Jensen  inequality,  yields
$$
 \E \big( (N_u^\delta)^2 \big) \leq K.
$$
This implies in turn that  the family  $N_u^\delta$ is uniformly integrable. As a consequence, $N_u^\delta$  converges also in $L^1$,  yielding 
 $$
 \E(N_u)  = \lim_{\delta \to 0}  \E(N_u ^\delta).
$$
Making explicit the  r.h.s. above we get directly \eqref{e:ricegen}. 
Under the integrability of $X'(t)$, the conditional expectation $\E\big(|X'(t)|\ \big| X(t)\big)$ is well defined  giving  \eqref{a:jma:2}.
\end{proof}

\begin{remark}
Analogous to the situation in Theorem \ref{t:jma:1}, the 
finiteness of the second moment of the level set can be obtained 
using Theorem \ref{t:codim}.
\end{remark}

\begin{example}[The case $D=d=1$]
Suppose that $X(t)$ satisfies the conditions of Theorem \ref{theorem:key} with and $k,m,h$ such that 
\begin{equation*}%\label{eq:nw}
p=2<\left(k-\frac{h}2-{1\over1+h}\right)\left(\frac1m+\frac1{1+h}\right)^{-1}.
\end{equation*}
In most  of the cases, it is very easy to see that the process also satisfies all the conditions of Theorem \ref{t:jma:1},
although this fact is not an exact consequence.
This is the case for all the examples considered in Section \ref{section:examples}.
Of course, for Gaussian processes and $\chi^2$ processes, the validity of Rice formula has been known for a long time.\bigskip 
\end{example}

\end{document}